\documentclass[12pt]{amsart}
\usepackage{amssymb}
\usepackage{bbm}
\usepackage{amsfonts}
\usepackage{latexsym}
\usepackage[all]{xy}
\usepackage{amscd}
\usepackage{comment}
\usepackage{fullpage}

\numberwithin{equation}{section}
\newtheorem{theorem}[equation]{Theorem}
\newtheorem{lemma}[equation]{Lemma}

\theoremstyle{definition}

\newcommand{\la}{\langle}
\newcommand{\ra}{\rangle}
\newcommand{\ot}{\otimes}

\newcommand{\lexp}[2]{{\vphantom{#2}}^{#1}{#2}}
\renewcommand{\k}{\Bbbk}

\newcommand{\ds}{\displaystyle}
\newcommand{\mH}{\mathcal{H}}
\newcommand{\kX}{\k \langle X \rangle}
\newcommand{\X}{\langle X \rangle}
\newcommand{\irr}{\text{irr}}

\DeclareMathOperator{\ddet}{det}

\begin{document}

\title[Quantum Drinfeld orbifold algebras]
{Quantum Drinfeld orbifold algebras}
\author{Piyush Shroff}
\email{piyushilashroff@gmail.com}
\address{Department of Mathematics, Texas State University, San Marcos, Texas 78666, USA}

\begin{abstract}
Quantum Drinfeld orbifold algebras are the generalizations of Drinfeld orbifold algebras, which are obtained by replacing polynomial rings by quantum polynomial rings. In [6], the authors give necessary and sufficient conditions on a defining parameters to obtain Drinfeld orbifold algebras. In this article we generalize their result. It also simultaneously generalizes the result of [5] about quantum Drinfeld Hecke algebras.
\end{abstract}

\maketitle

\begin{section}{Introduction}

Drinfeld orbifold algebras arise in different settings, for example, as Lusztig's graded affine Hecke algebras, symplectic reflection algebras, and rational Cherednik algebras. These algebras are deformations of skew group algebra generated by a finite group $G$ which acts on a polynomial ring over some vector space $V$.

\indent In [6], Shepler and Witherspoon considered the quotient of the skew group algebra $T(V)\# G$ (defined below), where $T(V)$ is the tensor algebra and $G$ is a finite group acting by linear transformations on a finite dimensional vector space $V$ over a field $\k$. They defined the resulting algebra to be a Drinfeld orbifold algebra if it satisfies the Poincar\'{e}-Birkhoff-Witt (PBW) property (defined below). The quotient is a deformation of skew group algebra $S\# G$ (defined below) where $S$ is the symmetric algebra with the induced action of $G$ by automorphisms. These kind of algebras where studied by Halbout, Oudom, and Tang [3], where a finite group $G$ acts faithfully on real vector space $V$. In this article we replace symmetric algebra by quantum symmetric algebra and express the conditions on algebra parameters in algebraic format to satisfy PBW property. For examples, we refer reader to [4], [5] and [6].

Let $\k$ be a field, and let $V$ be a finite-dimensional vector space over $\k$.
Let $v_1,v_2,\ldots,v_n$ be a basis for $V$, and let
${\bf q}:= (q_{ij})_{1 \leq i,j \leq n}$ be a tuple of
nonzero scalars for which $q_{ii}=1$ and $q_{ji}=q_{ij}^{-1}$ for all $i,j$.

Let $S_{\bf q}(V)$ denote the {\bf quantum symmetric algebra}:
\[
 S_{\bf q}(V) := \k\la v_1,\ldots,v_n \mid v_iv_j = q_{ij}v_jv_i  \mbox{ for all }
      1\leq i,j\leq n \ra.
\]

Let $G$ be a finite group acting linearly on $V$, and that there is an induced action on $S_{\bf q}(V)$ by algebra automorphisms.Then we may form the skew group algebra $S_{\bf q}(V)\# G$: Letting $A=S_{\bf q}(V)$, additively $A\# G$ is the free left $A$-module with basis $G$. We write $A\# G=\oplus_{g\in G} A_g$, where $A_g=\{a\# g\mid a\in A\}$, that is for each $a\in A$ and $g\in G$ we denote $a\# g\in A_g$ the $a$-multiple of $g$. Multiplication on $A\# G$ is determined by
$$(a\# g)(b\# h):=a(\lexp{g}b)\# gh$$
for all $a, b\in A$ and $g, h\in G$. Similarly we define $T(V)\# G$ where $G$ acts by automorphisms on the tensor algebra $T(V)$.\\
\indent Let $\kappa: V\times V\rightarrow (\k \oplus V) \ot \k G$ be a bilinear map for which 
$\kappa(v_i,v_j) = - q_{ij}\kappa(v_j,v_i)$ for all $1\leq i,j\leq n$, and $\{t_g\mid g\in G\}$ be a basis of the group algebra $\k G$. Define  
\[
  \mH_{\mathbf{q}, \kappa} := T(V)\# G/(  v_iv_j - q_{ij} v_jv_i - \kappa(v_i,v_j) \mid
     1\leq i,j\leq n),
\]
the quotient of the skew group algebra 
$T(V)\# G$ by the ideal generated by all elements of the form specified.
Giving each $v_i$ degree 1 and each group element $g$  degree 0, $\mH_{\mathbf{q}, \kappa}$ is a
filtered algebra. 
We say that $\mH_{\mathbf{q}, \kappa}$ satisfies the {\bf PBW condition} if 
one of the following equivalent conditions holds:
\begin{enumerate}
\item The associated graded algebra of $\mH_{\mathbf{q}, \kappa}$ is isomorphic to $S_{\bf q}(V)\# G $, as graded algebras.
\item The set $\{v_1^{m_1}v_2^{m_2}\cdots v_n^{m_n} t_g \mid m_i \geq 0, g \in G\}$ 
is a $\k$-basis for $\mH_{\mathbf{q}, \kappa}$.
\end{enumerate}

We will call $\mH_{\mathbf{q}, \kappa}$ a {\bf quantum Drinfeld orbifold algebra} if it satisfies
the PBW condition. In the case when all $q_{ij}=1$, these are the Drinfeld orbifold algebras studied by ~[6].

\end{section}
\begin{section}{Necessary and sufficient conditions}

For each $g\in G$, let $\kappa_g: V\times V\rightarrow \k \oplus V$ be the function determined by the
condition
\[
\kappa(v,w) = \sum_{g \in G} \kappa_g(v,w)t_g \qquad \text{ for all }  v, w \in V.
\]
Furthermore, let $\kappa_g^C: V\times V\rightarrow \k$ 
and $\kappa_g^L: V\times V\rightarrow V$ be the functions determined by the condition
\[
\kappa_g(v,w) = \kappa_g^C(v,w) + \kappa_g^L(v,w) \qquad \text{ for all }  v, w \in V
\]
where $\kappa_g^C$ and $\kappa_g^L$ are constant and linear parts of $\kappa_g$.
The condition $\kappa(v_i,v_j) = -q_{ij}\kappa(v_j,v_i)$ 
implies that $\kappa_g^C(v_i,v_j) = -q_{ij}\kappa_g^C(v_j,v_i)$ 
and $\kappa_g^L(v_i,v_j) = -q_{ij}\kappa_g^L(v_j,v_i)$ for each $g \in G$.

For each group element $g\in G$, let $g^j_i$ denote the
scalar determined by the equation 
\[
   \lexp{g}{v_j} = \sum_{i=1}^n g_i^j v_i.
\]
Define the {\bf quantum $(i,j,k,l)$-minor determinant} of $g$ as 
\[
    \ddet_{ijkl} (g) := g^j_lg^i_k - q_{ji} g^i_lg^j_k.
\]

The following lemma will be used in the proof of the Theorem 2.2 below.

\begin{lemma} 
\label{relation on det}
Let $g \in G$. We have:
\begin{enumerate}
\item[(i)] $q_{lk} \ddet_{ijkl}(g) = - \ddet_{ijlk}(g)$ for all $i,j,k,l$.
\item[(ii)] For each $i,j$, if $q_{ij}\neq 1$, then $g^i_kg^j_k=0$ for all $k$.
\end{enumerate}
\end{lemma}
\begin{proof}
Refer Lemma 3.2 and Corollary 3.6 in [4].
\end{proof}

In the proof of the following theorem, we will assume that the reader is
familiar with G.~Bergman's 1978 paper on the Diamond Lemma [1]. We can also approach the following proof using techniques of Braverman and Gaitsgory [2].

\begin{theorem}
The algebra $\mH_{\mathbf{q}, \kappa}$ is a quantum Drinfeld orbifold algebra if and only if
the following conditions hold:
\begin{enumerate}
\item For all $g,h \in G$ and $1 \leq i < j \leq n$,
\[
\kappa_g^C(v_j, v_i) = \sum_{k < l} \ddet_{ijkl}(h) \kappa_{hgh^{-1}}^C(v_l,v_k)
\quad \text{ and } \quad 
\lexp{h}{\left(\kappa_g^L(v_j, v_i)\right)} = \sum_{k < l} \ddet_{ijkl}(h) \kappa_{hgh^{-1}}^L(v_l,v_k).
\]
\text{For all distinct }$i,j,k$ \text{and for all} $g\in G$,\\
\item $q_{ji}q_{ki}v_i\kappa_g^L(v_k,v_j) - \kappa_g^L(v_k,v_j)\lexp{g}v_i - q_{kj}v_j\kappa_g^L(v_k,v_i) \\ \hspace*{1cm} + q_{ji}\kappa_g^L(v_k,v_i)\lexp{g}v_j + v_k\kappa_g^L(v_j,v_i) - q_{ki}q_{kj}\kappa_g^L(v_j,v_i) \lexp{g}v_k =0$\\\\
\item $\ds\sum_{h\in G}\{q_{ij}q_{ik}\kappa_{gh^{-1}}^L(\kappa_h^L(v_j,v_k),\lexp{h}v_i) - \kappa_{gh^{-1}}^L(v_i,\kappa_h^L(v_j,v_k)) +
q_{ik}q_{jk}\kappa_{gh^{-1}}^L(\kappa_h^L(v_k,v_i),\lexp{h}v_j)\\ \hspace*{1cm} - q_{ij}q_{ik}\kappa_{gh^{-1}}^L(v_j,\kappa_h^L(v_k,v_i))+ \kappa_{gh^{-1}}^L(\kappa_h^L (v_i,v_j), \lexp{h}v_k) - q_{ik}q_{jk}\kappa_{gh^{-1}}^L(v_k,\kappa_h^L (v_i,v_j))\} \\\\
=2\{\kappa_g^C(v_j,v_k)(v_i-q_{ij}q_{ik}\lexp{g}v_i)+\kappa_g^C(v_k,v_i)(q_{ij}q_{ik}v_j-q_{ik}q_{jk}\lexp{g}v_j)+\kappa_g^C(v_i,v_j)(q_{ik}q_{jk}v_k-\lexp{g}v_k)\}$\\\\
\item $\ds\sum_{h\in G}\{q_{ij}q_{ik}\kappa_{gh^{-1}}^C(\kappa_h^L(v_j,v_k),\lexp{h}v_i) - \kappa_{gh^{-1}}^C(v_i,\kappa_h^L(v_j,v_k)) +
q_{ik}q_{jk}\kappa_{gh^{-1}}^C(\kappa_h^L(v_k,v_i),\lexp{h}v_j)\\ \hspace*{0.5cm} - q_{ij}q_{ik}\kappa_{gh^{-1}}^C(v_j,\kappa_h^L(v_k,v_i))+ \kappa_{gh^{-1}}^C(\kappa_h^L (v_i,v_j), \lexp{h}v_k) - q_{ik}q_{jk}\kappa_{gh^{-1}}^C(v_k,\kappa_h^L (v_i,v_j))\}=0$\\
\end{enumerate}
\end{theorem}

\begin{proof}
We begin by expressing the algebra $\mH_{\mathbf{q}, \kappa}$
as a quotient of a free associative $\k$-algebra.
Let $X=\{v_1,v_2, \ldots, v_n\} \cup \{t_g \mid g \in G\}$, and 
let $\kX$ be the free associative $\k$-algebra generated by $X$.
Consider the reduction system 
\[
S=\{(t_gv_i, \lexp{g}{v_i}t_g), \, (t_gt_h, t_{gh}), \, 
(v_jv_i, q_{ji}v_iv_j + \kappa(v_j,v_i)) \mid g,h \in G, 1 \leq i < j \leq n\}
\]
for $\kX$. Let $I$ be the ideal of $\kX$ generated by the following elements:
\[
t_gv_i -\lexp{g}{v_i}t_g, \qquad t_gt_h - t_{gh}, \qquad 
v_jv_i-q_{ji}v_iv_j - \kappa(v_j,v_i), \qquad g,h \in G, 1 \leq i < j \leq n.
\]
In what follows, we will use the Diamond Lemma [2] to show that the set 
\[
\{v_1^{m_1}v_2^{m_2}\cdots v_n^{m_n}t_g \mid m_i \geq 0, g \in G\}
\]
is a $\k$-basis
for $\kX/I$ if and only if the four conditions in the statement of the theorem hold. 

Define a partial order $\leq$ on the free semigroup $\X$ as follows: First,
we declare that $v_1 < v_2 < \cdots < v_n < g$ for all $g \in G$, and then we set $A < B$ if 
\begin{enumerate}
\item[(i)] $A$ is of smaller length than $B$, or
\item[(ii)] $A$ and $B$ have the same length but $A$ is less than $B$ relative to
the lexicographic order.
\end{enumerate}
Then $\leq$ is a semigroup partial order on $\X$, compatible with the 
reduction system $S$, and having the descending chain condition. 
Thus, the hypothesis of the Diamond Lemma holds.

Observe that the set $\X_\irr$ of irreducible elements of $\X$ is precisely the
alleged $\k$-basis for $\kX/I$. That is,
\[
\X_\irr = \{v_1^{m_1}v_2^{m_2}\cdots v_n^{m_n}t_g  \mid m_i \geq 0, g \in G\}.
\]

In what follows, we show that all ambiguities of $S$ are
resolvable. The theorem will then follow by the Diamond Lemma.
There are no inclusion ambiguities, but there do exist overlap
ambiguities, and these correspond to the monomials 
\[
t_gt_ht_k, \qquad  t_gt_hv_i, \qquad t_hv_jv_i, \qquad v_kv_jv_i,  
\qquad \text{ where } 1 \leq i < j < k \leq n, \, g,h \in G.
\]

Associativity of the multiplication in $G$
implies that the ambiguity corresponding to the monomial $t_gt_ht_k$ is resolvable.
The equality $\lexp{gh}{v_i} = \lexp{g}{(\lexp{h}{v_i})}$
implies that the ambiguity corresponding to the monomial $t_gt_hv_i$ is resolvable. 
Next, we show that the ambiguity
corresponding to the monomial $t_hv_jv_i$ is resolvable if and only if condition (1) 
in the statement of the theorem holds. Below, we use the symbol ``$\longrightarrow$" 
to indicate that a reduction has been applied. We have
{\tiny
\begin{eqnarray*}
t_hv_jv_i &\longrightarrow& q_{ji}t_hv_iv_j + t_h\kappa(v_j, v_i)\\
&\longrightarrow& q_{ji} \lexp{h}{v_i}\lexp{h}{v_j}t_h + t_h \kappa(v_j,v_i)\\
&=& q_{ji} \left(\sum_{l=1}^n h_l^i v_l \right) \left(\sum_{k=1}^n h_k^j v_k \right)t_h
+ t_h \kappa(v_j,v_i)\\
&=& q_{ji} \sum_{l<k} h_l^ih_k^j v_lv_kt_h + q_{ji}\sum_{k<l} h_l^ih_k^j v_lv_kt_h 
+ q_{ji} \sum_{k=1}^n h_k^ih_k^jv_k^2t_h + t_h \kappa(v_j,v_i)\\
&\longrightarrow& q_{ji} \sum_{l<k} h_l^ih_k^j v_lv_kt_h + 
q_{ji}\sum_{k<l} h_l^ih_k^j q_{lk} v_k v_lt_h 
+ q_{ji} \sum_{k<l}h_l^ih_k^j \kappa(v_l,v_k)t_h 
+ q_{ji} \sum_{k=1}^n h_k^ih_k^jv_k^2t_h + t_h \kappa(v_j,v_i) \\
&\longrightarrow& q_{ji} \sum_{k<l} \left(h_k^ih_l^j + q_{lk}h_l^ih_k^j\right) v_kv_lt_h 
+  q_{ji} \sum_{k=1}^n h_k^ih_k^jv_k^2t_h
+  q_{ji} \sum_{g \in G} \left( \sum_{k<l}h_l^ih_k^j \kappa_g(v_l,v_k)\right)t_{gh} 
+  \sum_{g \in G} t_h\kappa_g(v_j,v_i) t_g\\
&\longrightarrow& q_{ji} \sum_{k<l} \left(h_k^ih_l^j + q_{lk}h_l^ih_k^j\right) v_kv_lt_h 
+  q_{ji} \sum_{k=1}^n h_k^ih_k^jv_k^2t_h
+  q_{ji} \sum_{g \in G} \left( \sum_{k<l}h_l^ih_k^j \kappa_g^C(v_l,v_k)\right)t_{gh} \\
&& +  \; q_{ji} \sum_{g \in G} \left( \sum_{k<l}h_l^ih_k^j \kappa_g^L(v_l,v_k)\right)t_{gh} 
+  \sum_{g \in G} \kappa_g^C(v_j,v_i) t_{hg}
+  \sum_{g \in G} \lexp{h}{\left(\kappa_g^L(v_j,v_i)\right)} t_{hg}\\
&=& q_{ji} \sum_{k<l} \left(h_k^ih_l^j + q_{lk}h_l^ih_k^j\right) v_kv_lt_h 
+  q_{ji} \sum_{k=1}^n h_k^ih_k^jv_k^2t_h
+  \sum_{g \in G} \left( q_{ji} \sum_{k<l}h_l^ih_k^j \kappa_{hgh^{-1}}^C(v_l,v_k) 
+   \kappa_g^C(v_j,v_i)\right) t_{hg}\\
&& + \;  \sum_{g \in G} \left( q_{ji} \sum_{k<l}h_l^ih_k^j \kappa_{hgh^{-1}}^L(v_l,v_k) 
+   \lexp{h}{\left(\kappa_g^L(v_j,v_i)\right)}\right) t_{hg}
\end{eqnarray*}
}
and
{\tiny
\begin{eqnarray*}
t_hv_jv_i &\longrightarrow& \lexp{h}{v_j}\lexp{h}{v_i}t_h\\
&=& \left(\sum_{l=1}^n h_l^j v_l \right) \left(\sum_{k=1}^n h_k^i v_k \right)t_h \\
&=& \sum_{l<k} h_l^jh_k^i v_lv_kt_h + \sum_{k<l} h_l^jh_k^i v_lv_kt_h 
+ \sum_{k=1}^n h_k^jh_k^iv_k^2t_h\\
&\longrightarrow& \sum_{l<k} h_l^jh_k^i v_lv_kt_h + 
\sum_{k<l} q_{lk} h_l^jh_k^i v_kv_lt_h + \sum_{k<l}h_l^jh_k^i \kappa(v_l,v_k)t_h
+ \sum_{k=1}^n h_k^jh_k^iv_k^2t_h \\
&\longrightarrow& \sum_{k<l} \left(h_k^jh_l^i + q_{lk}h_l^jh_k^i\right) v_kv_lt_h 
+ \sum_{k=1}^n h_k^jh_k^iv_k^2t_h 
+ \sum_{g \in G} \left(  \sum_{k<l}h_l^jh_k^i \kappa_g(v_l,v_k) \right) t_{gh}\\
&=& \sum_{k<l} \left(h_k^jh_l^i + q_{lk}h_l^jh_k^i\right) v_kv_lt_h 
+ \sum_{k=1}^n h_k^jh_k^iv_k^2t_h 
+ \sum_{g \in G} \left(  \sum_{k<l}h_l^jh_k^i \kappa_{hgh^{-1}}^C(v_l,v_k) \right) t_{hg}
+ \sum_{g \in G} \left(  \sum_{k<l}h_l^jh_k^i \kappa_{hgh^{-1}}^L(v_l,v_k) \right) t_{hg}.
\end{eqnarray*}
}

The last expressions in the previous two computations are equal if and only if

\begin{enumerate}
\item[(a)] $q_{ji}h_k^ih_l^j + q_{ji}q_{lk}h_l^ih_k^j = h_k^jh_l^i + q_{lk}h_l^jh_k^i$ for all $k<l$,
\item[(b)] $q_{ji} h_k^ih_k^j = h_k^ih_k^j$ for all $k$, and
\item[(c)] for all $g \in G$, we have
\[
\ds  q_{ji} \sum_{k<l}h_l^ih_k^j \kappa_{hgh^{-1}}^C(v_l,v_k) 
+   \kappa_g^C(v_j,v_i)\\
= \sum_{k<l}h_l^jh_k^i \kappa_{hgh^{-1}}^C(v_l,v_k)
\]
and
\[
\ds  q_{ji} \sum_{k<l}h_l^ih_k^j \kappa_{hgh^{-1}}^L(v_l,v_k) 
+   \lexp{h}{\left(\kappa_g^L(v_j,v_i)\right)}\\
= \sum_{k<l}h_l^jh_k^i \kappa_{hgh^{-1}}^L(v_l,v_k).
\]
\end{enumerate}

That (a) and (b) hold follows from part (i) and part (ii) of Lemma~\ref{relation on det},
respectively. The equations in (c) are equivalent to the equations in 
condition (1) in the statement of the theorem.

Finally, we show that the ambiguity corresponding to the monomial $v_kv_jv_i$ is resolvable
if and only if conditions (2)-(4) in the statement of the theorem hold.
{\tiny
\begin{eqnarray*}
v_kv_jv_i &\longrightarrow& q_{ji} v_kv_iv_j + v_k \kappa(v_j,v_i) \\
&=& q_{ji} v_kv_iv_j + \sum_{g \in G} \left( v_k \kappa_g^C(v_j,v_i) t_g + 
v_k \kappa_g^L(v_j,v_i) t_g \right) \\
&\longrightarrow& q_{ji} (q_{ki} v_iv_kv_j + \kappa(v_k,v_i) v_j) + 
\sum_{g \in G} \left( \kappa_g^C(v_j,v_i) v_k t_g + v_k \kappa_g^L(v_j,v_i) t_g \right) \\
&=& q_{ji} q_{ki} v_iv_kv_j + 
q_{ji} \sum_{g \in G} \left( \kappa_g^C(v_k,v_i) t_g v_j + \kappa_g^L(v_k,v_i) t_g v_j \right) + 
\sum_{g \in G} \left( \kappa_g^C(v_j,v_i) v_k t_g + v_k \kappa_g^L(v_j,v_i) t_g \right) \\
&\longrightarrow& q_{ji} q_{ki} v_iv_kv_j + 
q_{ji} \sum_{g \in G} \left( \kappa_g^C(v_k,v_i) \lexp{g}{v_j} t_g + \kappa_g^L(v_k,v_i) \lexp{g}{v_j} t_g \right) + 
\sum_{g \in G} \left( \kappa_g^C(v_j,v_i) v_k t_g + v_k \kappa_g^L(v_j,v_i) t_g \right) \\
&\longrightarrow& q_{ji} q_{ki} (q_{kj} v_iv_jv_k + v_i \kappa(v_k,v_j)) + 
q_{ji} \sum_{g \in G} \left( \kappa_g^C(v_k,v_i) \lexp{g}{v_j} t_g + \kappa_g^L(v_k,v_i) \lexp{g}{v_j} t_g \right) \\
&& + \sum_{g \in G} \left( \kappa_g^C(v_j,v_i) v_k t_g + v_k \kappa_g^L(v_j,v_i) t_g \right) \\
&=& q_{ji} q_{ki} q_{kj} v_iv_jv_k + q_{ji} q_{ki} \sum_{g \in G} \left( v_i \kappa_g^C(v_k,v_j) t_g + v_i \kappa_g^L(v_k,v_j) t_g \right) + 
q_{ji} \sum_{g \in G} \left( \kappa_g^C(v_k,v_i) \lexp{g}{v_j} t_g + \kappa_g^L(v_k,v_i) \lexp{g}{v_j} t_g \right) \\ 
&& + \sum_{g \in G} \left( \kappa_g^C(v_j,v_i) v_k t_g + v_k \kappa_g^L(v_j,v_i) t_g \right) \\
&=& q_{ji} q_{ki} q_{kj} v_iv_jv_k + q_{ji} q_{ki} \sum_{g \in G} \left( \kappa_g^C(v_k,v_j) v_i t_g +
v_i \kappa_g^L(v_k,v_j) t_g \right) + 
q_{ji} \sum_{g \in G} \left( \kappa_g^C(v_k,v_i) \lexp{g}{v_j} t_g + \kappa_g^L(v_k,v_i) \lexp{g}{v_j} t_g \right) \\ 
&& + \sum_{g \in G} \left( \kappa_g^C(v_j,v_i) v_k t_g + v_k \kappa_g^L(v_j,v_i) t_g \right)\\\\
\normalsize {\text{and}}\\\\
v_kv_jv_i&\longrightarrow& q_{kj} v_jv_kv_i + \kappa(v_k,v_j) v_i \\
&=& q_{kj} v_jv_kv_i + \sum_{g \in G} \left( \kappa_g^C(v_k,v_j) t_g v_i + \kappa_g^L(v_k,v_j) t_g v_i \right) \\
&\longrightarrow& q_{kj} v_jv_kv_i + \sum_{g \in G} \left( \kappa_g^C(v_k,v_j) \lexp{g}{v_i} t_g + \kappa_g^L(v_k,v_j) \lexp{g}{v_i} t_g \right) \\
&\longrightarrow& q_{kj} (q_{ki} v_jv_iv_k + v_j \kappa(v_k,v_i)) + 
\sum_{g \in G} \left( \kappa_g^C(v_k,v_j) \lexp{g}{v_i} t_g + \kappa_g^L(v_k,v_j) \lexp{g}{v_i} t_g \right)
\end{eqnarray*}
}
{\tiny
\begin{eqnarray*}
v_kv_jv_i&=& q_{kj} q_{ki} v_jv_iv_k + 
q_{kj} \sum_{g \in G} \left( v_j \kappa_g^C(v_k,v_i) t_g + v_j \kappa_g^L(v_k,v_i) t_g \right) + 
\sum_{g \in G} \left( \kappa_g^C(v_k,v_j) \lexp{g}{v_i} t_g + \kappa_g^L(v_k,v_j) \lexp{g}{v_i} t_g \right) \\
&\longrightarrow& q_{kj} q_{ki} (q_{ji} v_iv_jv_k + \kappa(v_j,v_i) v_k) + 
q_{kj} \sum_{g \in G} \left( \kappa_g^C(v_k,v_i) v_j t_g + v_j \kappa_g^L(v_k,v_i) t_g \right) \\ 
&& + \sum_{g \in G} \left( \kappa_g^C(v_k,v_j) \lexp{g}{v_i} t_g + \kappa_g^L(v_k,v_j) \lexp{g}{v_i} t_g \right) \\
&=& q_{kj} q_{ki} q_{ji} v_iv_jv_k + 
q_{kj} q_{ki} \sum_{g \in G} \left( \kappa_g^C(v_j,v_i) t_g v_k + \kappa_g^L(v_j,v_i) t_g v_k \right) + 
q_{kj} \sum_{g \in G} \left( \kappa_g^C(v_k,v_i) v_j t_g + v_j \kappa_g^L(v_k,v_i) t_g \right) \\ 
&& + \sum_{g \in G} \left( \kappa_g^C(v_k,v_j) \lexp{g}{v_i} t_g + \kappa_g^L(v_k,v_j) \lexp{g}{v_i} t_g \right) \\
&\longrightarrow& q_{kj} q_{ki} q_{ji} v_iv_jv_k + 
q_{kj} q_{ki} \sum_{g \in G} \left( \kappa_g^C(v_j,v_i) \lexp{g}{v_k} t_g + \kappa_g^L(v_j,v_i) \lexp{g}{v_k} t_g \right) + 
q_{kj} \sum_{g \in G} \left( \kappa_g^C(v_k,v_i) v_j t_g + v_j \kappa_g^L(v_k,v_i) t_g \right) \\ 
&& + \sum_{g \in G} \left( \kappa_g^C(v_k,v_j) \lexp{g}{v_i} t_g + \kappa_g^L(v_k,v_j) \lexp{g}{v_i} t_g \right).\\\\
\end{eqnarray*}
}
The last expressions in the previous two computations are equal if and only if\\
{\tiny
\begin{eqnarray*}
\ds\sum_{g\in G}\{\kappa_g^C(v_k,v_j)(q_{ji}q_{ki}v_i-\lexp{g}v_i)+\kappa_g^C(v_i,v_k)
(q_{kj}q_{ki}v_j-q_{ji}q_{ki}\lexp{g}v_j)+\kappa_g^C(v_j,v_i)(v_k-q_{kj}q_{ki}\lexp{g}v_k)+ q_{ji}q_{ki}v_i\kappa_g^L(v_k,v_j)\\-\kappa_g^L(v_k,v_j)\lexp{g}v_i+q_{kj}q_{ki}v_j
\kappa_g^L(v_i,v_k)-q_{ji}q_{ki}\kappa_g^L(v_i,v_k)\lexp{g}v_j+v_k\kappa_g^L(v_j,v_i)-q_{kj}q_{ki}
\kappa_g^L(v_j,v_i)\lexp{g}v_k\}=0\\
\end{eqnarray*}
}
\noindent Write $\lexp{g}v_a=\ds\sum_l g_l^av_l$ and $\kappa_g^L(v_a,v_b)=\ds\sum_m C_m^{g,a,b}v_m$ for some $g_l^a$ and $C_m^{g,a,b}$ in $\k$. Now similar calculations as in proof of Theorem 3.1 [6] using the Diamond lemma will lead us to conditions (2)-(4) in the statement of the theorem.\\
\end{proof}
\vspace*{.25cm}

\noindent{\bf Acknowledgement:} The author would like to thank Dr. Deepak Naidu and Dr. Jeanette Shakalli for their initial help and Dr. Anne Shepler for her helpful comments. In addition, the author would like to thank his doctoral thesis advisor Dr. Sarah Witherspoon.

\end{section}




\begin{thebibliography}{AB}

\bibitem{[1]} G. Bergman, 
\textit{The diamond lemma for ring theory}, 
Adv. in Math. \textbf{29} (1978), 178--218.

\bibitem{[2]} A. Braverman and D. Gaitsgory,
\textit{Poincar\'{e}-Birkoff-Witt Theorem for quadratic algebras of Koszul type},
J. Algebra \textbf{181} (1996), 315--328.

\bibitem{[3]} G. Halbout, J.-M. Oudom, and X. Tang,
\textit{Deformations of orbifolds with noncommutative linear Poisson structures},
Int. Math. Res. Not. (2011), no. 1, 1--39.

\bibitem{[4]} V. Levandovskyy and A. V. Shepler,
\textit{Quantum Drinfeld Hecke algebras},
\texttt{arXiv:1111.4975v2}.

\bibitem{[5]} D. Naidu and S. Witherspoon,
\textit{Hochschild cohomology and quantum Drinfeld Hecke algebras},
\texttt{arXiv:1111.5243v1}.

\bibitem{[6]} A. V. Shepler and S. J. Witherspoon,
\textit{Drinfeld orbifold algebras}, 
Pacific J. Math \textbf{259-1} (2012), 161--193.\\\\


\end{thebibliography}
\end{document}